\documentclass[14pt]{amsart}

\usepackage{todonotes}

\usepackage{amsfonts, amssymb, amscd}

\usepackage[symbol]{footmisc}
\usepackage{amsmath}
\usepackage{bm}
\usepackage{verbatim}
\usepackage{mathrsfs}
\usepackage{lineno}
\usepackage{graphicx}
\usepackage{tikz-cd}
\usepackage{subcaption}
\usepackage{listings}
\usepackage{subfiles}
\usepackage[toc,page]{appendix}
\usepackage{mathtools}
\usepackage{comment}
\usepackage{enumerate}
\usepackage{enumitem}
\usepackage[all]{xy}

\usepackage{graphicx}
\graphicspath{{images/}}

\usepackage{appendix}
\usepackage{hyperref}
\hypersetup{
	colorlinks=true,
	citecolor=red,
	linkcolor=blue,
	filecolor=magenta,      
	urlcolor=red,
}
\newcommand{\Cd}{\mathcal{D}}
\newcommand{\Ci}{\mathcal{I}}

\newcommand{\Co}{\mathcal{O}}
\newcommand{\Cr}{\mathcal{R}}

\newcommand{\Cx}{\mathcal{X}}

\newcommand{\Cb}{\mathcal{B}}
\newcommand{\Ce}{\mathcal{E}}
\newcommand{\Cw}{\mathcal{W}}

\newcommand{\Ct}{\mathcal{T}}

\newcommand{\Cc}{\mathcal{C}}
\newcommand{\Ch}{\mathcal{H}}

\newcommand{\Supp}{\mathrm{Supp}}
\newcommand{\vol}{\mathrm{vol}}
\newcommand{\hvol}{\widehat{\mathrm{vol}}}

\theoremstyle{plain} 
\newtheorem{thm}{Theorem}[section] 
\newtheorem{lemma}[thm]{Lemma}

\theoremstyle{definition} 
\newtheorem{defn}[thm]{Definition} 

\theoremstyle{remark} 

\begin{document}
	\title{Discreteness of volumes of divisors on Calabi--Yau type varieties}
	
	\author{Junpeng Jiao}
	\email{jiao$\_$jp@tsinghua.edu.cn}
	\address{Yau Mathematical Sciences Center, Tsinghua University, Beijing, China}
	

	\begin{abstract}
		We study the volumes of divisors in Calabi--Yau type varieties. We show that given a klt Calabi--Yau pair $(X,B)$ and an integral divisor $A$ on $X$, the volume of $A$ is in a discrete set depending only on the dimension and singularities of $(X,B)$. As an application, we prove a boundedness result of polarized log Calabi--Yau pairs which was conjectured by Birkar.
	\end{abstract}
	
	\maketitle
	\tableofcontents
	Throughout this paper, we work over $\mathbb{C}$. 
	\section{Introduction}
	The study of Calabi-–Yau pairs--projective varieties equipped with boundary divisors satisfying the Calabi–Yau condition $K_X+B\sim_{\mathbb{Q}} 0$--plays a central role in the minimal model program and birational geometry. A fundamental problem is understanding the distribution of volumes of divisors and the boundedness of such pairs under natural geometric constraints. In this paper, we establish two main results addressing these questions.
	
	Our first theorem demonstrates the discreteness of volumes of integral divisors on $\epsilon$-log canonical ($\epsilon$-lc) Calabi–Yau pairs.
	\begin{thm}\label{Main theorem: discreteness of volume}
		Fix $d\in \mathbb{N}$ and $\epsilon\in (0,1)$, then there exists a discrete set $\Cc$ depending only on $d,\epsilon$ such that for every $d$-dimensional $\epsilon$-lc Calabi--Yau pair $(X,B)$, if $A$ is an integral divisor on $X$, then 
		$$\vol(A)\in \Cc.$$
	\end{thm}
	This discreteness arises purely from the underlying geometry, independent of the coefficients of the boundary divisor $B$.
	
	Our second result achieves boundedness for polarized Calabi–Yau pairs without requiring the coefficients of $B$ to lie in a finite set. This result was conjectured by Birkar.
	\begin{thm}\label{Main theorem: boundedness of polarized Calabi--Yau}
		Fix $d\in \mathbb{N}$, $\epsilon\in (0,1)$, and $v>0$, then the set of projective varieties $X$ such that
		\begin{itemize}
			\item $\mathrm{dim}(X)=d$,
			\item $(X,B)$ is $\epsilon$-lc Calabi--Yau for some $\mathbb{Q}$-divisor $B$ on $X$, and
			\item there exists an ample integral divisor $A$ on $X$ such that $A^d\leq v$.
		\end{itemize}
		forms a bounded family.
	\end{thm}

	\section{Preliminaries}
	We will use the same notation as in \cite{KM98} and \cite{Laz04}.

	A \textbf{couple} consists of a normal variety $X$ and an effective $\mathbb{Q}$-divisor $B$ on $X$. A couple is called a \textbf{pair} if further $\mathrm{coeff}(B)\subset [0,1]$ and $K_X+B$ is $\mathbb{Q}$-Cartier. For a number $\epsilon\in (0,1)$, a pair $(X,B)$ is called $\epsilon$-\textbf{log canonical} ($\epsilon$-\textbf{lc}) if the total log discrepancy $\text{tld}(X,B)\geq \epsilon$.

	Let $(X,B),(Y,B_Y)$ be two sub-pairs and $h:Y\rightarrow X$ a projective birational morphism, we say $(Y,B_Y)\rightarrow (X,B)$ is a \textbf{crepant birational morphism} if $B_Y$ is the $\mathbb{Q}$-divisor such that $K_Y+B_Y\sim_{\mathbb{Q}}h^*(K_X+B)$. 
	A crepant birational morphism $(Y,B_Y)\rightarrow (X,B)$ is called a $\mathbb{Q}$-\textbf{factorization} if $Y\rightarrow X$ is small and $Y$ is $\mathbb{Q}$-factorial. Two sub-pairs $(X_1,B_1)$ and $(X_2,B_2)$ are \textbf{crepant birationally equivalent} if and only if there exists a sub-pair $(Y,B_Y)$ and two crepant birational morphisms $(Y,B_Y)\rightarrow (X_i,B_i),i=1,2$.
	
	Given two $\mathbb{Q}$-divisors $A,B$, $A\sim_{\mathbb{Q}} B$ means that there exists $m\in \mathbb{N}$ such that $m(A-B)\sim 0$. We say $A\in |B|_{\mathbb{Q}}$ if $A\sim_{\mathbb{Q}} B$ and $A\geq 0$. Let $X$ be a projective variety and $D$ a $\mathbb{Q}$-divisor on $X$, we use $\Co_X(mD)$ to denote the divisorial sheaf $\Co_X(\lfloor mD \rfloor)$.
	
	Let $X\rightarrow T$ be a contraction with generic fiber $X_\eta$ and $B$ a $\mathbb{Q}$-Cartier $\mathbb{Q}$-divisor on $X$. For a point $t\in T$ or an open subset $U\subset T$, we denote $B|_{X_t}$ by $B_t$, $B|_{X_U}$ by $B_U$, and $B|_{X_\eta}$ by $B_\eta$ for simplicity of notation.

	\subsection{Log boundedness}
	\begin{defn}
		Let $\mathscr{S}$ be a set of $d$-dimensional couples. We say $\mathscr{S}$ is \textbf{bounded} if there exists $v>0$ such that for every $(X,B)\in \mathscr{S}$, there exists a very ample divisor $A$ on $X$ such that $A^d\leq v$. We say $\mathscr{S}$ is \textbf{log bounded }if further there exists $c\in \mathbb{Q}_+\cap (0,1)$ such that $B=cD$ for an integral divisor $D$ and $A-B$ is pseudo-effective. 
	\end{defn}
	
	By the boundedness of the Chow variety, see \cite[\S 1.3]{Kol96}, a set of $d$-dimensional couples $\mathscr{S}$ is bounded (resp. log bounded) if and only if there exists a flat morphism $\Cx\rightarrow T$ of relative dimension $d$ (resp. a flat morphism $\Cx\rightarrow T$ of relative dimension $d$ with an effective $\mathbb{Q}$-divisor $\Cb$ on $\Cx$ whose support has codimension one on every fiber) over a scheme of finite type, such that for every $(X,B)\in \mathscr{S}$, there exists a closed point $t\in T$ and an isomorphism $X\cong \Cx_t$ (resp. $(X,B)\cong (\Cx_t,\Cb_t)$). We call $\Cx\rightarrow T$ (resp. $(\Cx,\Cb)\rightarrow T$) a \textbf{parametrizing family} of $\mathscr{S}$.

	\subsection{Boundedness of Cartier index}
	We need to following definition and result on boundedness of Cartier index of integral divisors from \cite{HLQZ25}, for more relative results, see \cite{XZ21}.
	\begin{defn}[{\cite[3.2]{HLQZ25}}]
		For a $\mathbb{Q}$-Cartier $\mathbb{Q}$-divisor $D$ on a projective normal variety $X$, a graded linear series $V_\bullet$ is a sequence of subspaces $V_m\in H^0(X,\mathcal{O}_X(mD)),m\geq 0$ such that $V_l\cdot V_m\subset V_{m+l}$ for every $m,l$. Its volume is defined as 
		$$\vol(V_\bullet):=\limsup_{m\rightarrow +\infty} \frac{\mathrm{dim}(V_m)}{m^d/d!},$$
		where $d=\mathrm{dim}(X).$
		
		We say that a graded linear series $V_\bullet$ is eventually birational, if the rational map $X\dashrightarrow \mathbb{P}(V_m^\wedge)$ induced by the linear system $|V_m|$ is birational onto its image for all sufficiently large $m$ such that $V_m\neq \emptyset$. 
		
		Let $(X,B)$ be a projective pair and $D$ a big $\mathbb{Q}$-divisor such that $(X,B+D')$ is klt for some $D'\in |D|_{\mathbb{Q}}$. A graded linear series $V_\bullet$ of $D$ is called admissible with respect to $(X,B)$ if it is eventually birational and $(X,B+\Gamma)$ is lc for all $m$ such that $V_m\neq \emptyset$ and $\Gamma\in \frac{1}{m}V_m$.
		We then define the log canonical volume of $D$ with respect to $(X,B)$, denoted by $\hvol_{X,B}(D)$, to be
		\begin{equation*}
			\begin{aligned}
				\hvol_{X,B}(D):=\sup \{\vol(V_\bullet)\mid V_\bullet\text{ is an admissible graded linear series of }D\\
				\text{ with respect to }(X,B)\}.
			\end{aligned}
		\end{equation*}
		According to \cite[Lemma 3.5]{HLQZ25}, we have $\hvol_{X,B}(D)>0$.
	\end{defn}

	\begin{lemma}\label{boundedness of Cartier index by lc volume}
		Let $(X,B)$ be a projective $d$-dimensional pair and $D$ a big $\mathbb{Q}$-divisor such that $(X,B+D')$ is klt for some $D'\in|D|_{\mathbb{Q}}$.
	 	Then the Cartier index of any $\mathbb{Q}$-Cartier integral divisor on $X$ is at most $\lfloor \frac{d^d}{\hvol_{X,B}(D)}\rfloor !.$
	 	\begin{proof}
	 		This follows directly from \cite[Lemma 3.13 and Lemma 3.14]{HLQZ25}.
	 	\end{proof}
	\end{lemma}
	\begin{lemma}\label{boundedness of Cartier index by volume and lct}
		Fix $d\in \mathbb{N}$, $t\in (0,1)$, and $\alpha>0$, then there exists $r\in\mathbb{N}$ depending only on $d,t,v$ such that:
		
		If $(X,B)$ is a projective klt pair, $D$ a big $\mathbb{Q}$-divisor such that $\vol(D)\geq \alpha$ and $(X,B+tD')$ is klt for every $D'\in |D|_\mathbb{Q}$. Then for every $\mathbb{Q}$-Cartier integral divisor $R$ on $X$, $rR$ is Cartier.
		\begin{proof}
			Because $(X,B+tD')$ is klt for every $D'\in |D|_\mathbb{Q}$, then $V_m:=H^0(X,\mathcal{O}_X(mtD))$ is an admissible graded linear series of $tD$ with respect to $(X,B)$ and $\hvol_{X,B}(tD)=\vol(tD)\geq t^d\alpha$. Then by Lemma \ref{boundedness of Cartier index by lc volume}, we only need to choose 
			$$r:=\lfloor \frac{d^d}{t^d\alpha}\rfloor !.$$
		\end{proof}
	\end{lemma}

	\section{Proof of main results}
		A projective lc pair $(X,B)$ is called a \textbf{Calabi--Yau pair} if $K_X+B\sim_{\mathbb{Q}} 0$.
		
		We say a projective variety $X$ is of \textbf{Calabi--Yau type} if there exists a $\mathbb{Q}$-divisor $B$ such that $(X,B)$ is a Calabi--Yau pair, we say $X$ is $\epsilon$-\textbf{lc Calabi--Yau type} if $(X,B)$ is a $\epsilon$-lc Calabi--Yau pair. 
		
		Let $X$ be a projective variety of Calabi--Yau type, a $\mathbb{Q}$-divisor $B$ such that 
		$(X,B)$ is a Calabi--Yau pair (resp. a klt Calabi--Yau pair) is called a $\mathbb{Q}$-\textbf{complement} (resp. \textbf{klt $\mathbb{Q}$-complement}) of $X$. We also call $B$ an \textbf{$m$-complement} of $X$ if further $m(K_X+B)\sim 0$ for some $m\in \mathbb{N}$. See \cite{PS09}, \cite{Bir19}, and \cite{Sho20} for more results on complements.
		
		Assume $\mathscr{S}$ is a bounded set of Calabi--Yau type varieties, we are interested in the following questions.
		\begin{enumerate}
			\item Does there exists $m\in \mathbb{N}$ depending only on $\mathscr{S}$ such that every $X\in \mathscr{S}$ has an $m$ complement?
			\item Fix $\epsilon\in (0,1)$, does there exists $m\in \mathbb{N}$ depending only on $\mathscr{S}$ and $\epsilon$ such that if $X\in \mathscr{S}$ is $\epsilon$-lc Calabi--Yau type, then $X$ has a klt $m$-complement?
			\item Fix $\epsilon\in (0,1)$ and $v>0$, does there exists a finite set $\Ci\subset \mathbb{R}_+$ depending only on $\mathscr{S},\epsilon$ and $v$ such that if $X\in \mathscr{S}$ is $\epsilon$-lc Calabi--Yau type and $D$ is an integral divisor on $X$ with $\mathrm{deg}(D)\leq v$, then $\vol(D)\in \Ci$.
		\end{enumerate}
		In this paper we give an affirmative answer to question (3) by bounding the Cartier index on the canonical model. Note if question (2) is true, then question (3) follows directly from the log version of Invariance of plurigenera, see \cite[Theorem 1.8]{HMX13}.
		
		\begin{defn}
			Fix $d,v\in \mathbb{N}$ and $\epsilon\in (0,1)$. We define $\mathscr{C}(d,v,\epsilon)$ to be the set of $d$-dimensional projective couples $(X,A)$ such that
			\begin{itemize}
				\item $X$ is $\epsilon$-lc Calabi--Yau type, and
				\item $A$ is an ample integral divisor on $X$ such that $\vol(A)\leq v$.
			\end{itemize}
		\end{defn}
		
		According to \cite[Theorem 1.3]{JJZ25}, $\mathscr{C}(d,v,\epsilon)$ is bounded in codimension one. Suppose $X\dashrightarrow X'$ is an isomorphism in codimension one, it is easy to see that $X$ is Calabi--Yau type if and only if $X'$ is Calabi--Yau type. Furthermore, if $\Delta$ is an $m$-complement of $X$, then its strict transform $\Delta'$ on $X'$ is an $m$-complement of $X'$ and $(X,\Delta)$ is crepant birationally equivalent to $(X',\Delta')$.
		Thus, question (2) predicts that there exists $m$ depending only on $d,v,\epsilon$ such that every $X\in \mathscr{C}(d,v,\epsilon)$ has a klt $m$-complement and $\mathscr{C}(d,v,\epsilon)$ is bounded according to \cite[Corollary 1.8]{Bir23}.
		However, constructing bounded complements—particularly klt ones—is challenging, making it difficult to prove boundedness via complements.
		
		First, we use Theorem \ref{boundedness of pseudo-effective threshold} and Theorem \ref{boundedness of polarized CY in codimension 1} to show that $\mathscr{C}(d,v,\epsilon)$ is not only bounded, but also log bounded in codimension one.
		
		\begin{thm}\label{boundedness of pseudo-effective threshold}
			Fix $d,v\in\mathbb{N}$ and $\epsilon\in (0,1)$. Then there exists $\delta>0$ satisfying the following: 
			
			If $(X,A)\in \mathscr{C}(d,v,\epsilon)$ is a couple, then $K_X+\delta A$ is not pseudo-effective.
			\begin{proof}
				Suppose the result is not true, then we have a sequence of $\epsilon$-lc Calabi--Yau pairs $(X_i,B_i),i\in \mathbb{N}$ with integral ample divisors $A_i$ such that $A_i^d\leq v$, and a sequence of descending positive number $\delta_i\rightarrow 0$ such that $K_{X_i}+\delta_iA_i$ is big.
				
				After taking a $\mathbb{Q}$-factorization, we may assume $X_i$ is $\mathbb{Q}$-factorial. By \cite[Theorem 1.1]{Bir23}, there exists a fixed $m\in\mathbb{N}$ such that $|mA_i|$ defines a birational map $X_i\dashrightarrow W_i$. Let $h_i:Y_i\rightarrow X_i,g_i:Y_i\rightarrow W_i$ be a common resolution, then 
				$$h_i^*(mA_i)=g_i^*H_i+F_i,$$ 
				where $H_i$ is a very ample divisor on $W_i$ and $F_i$ is effective. Because $\mathrm{vol}(A_i)\leq v$, then $\mathrm{vol}(H_i)\leq m^dv$. By the boundedness of Chow varieties, $W_i$ is in a bounded family $\Cw\rightarrow \Ct$. We assume $W_i$ is isomorphic to $\Cw_{t_i}$ for a closed point $t_i\in \Ct$ and $H_i=\Ch_{t_i}:=\Ch|_{\Cw_{t_i}}$, where $\Ch$ is the corresponding relatively very ample divisor on $\Cw$ over $\Ct$.
				
				After passing to a stratification of $\Ct$, we may assume $\Cw\rightarrow \Ct$ has a fiberwise log resolution $f:\Cw'\rightarrow \Cw$.  Because $\Ch$ is relatively very ample, there exist $r\in\mathbb{N}$ and a relatively very ample divisor $\Ch'$ on $\Cw'$ over $\Ct$ such that 
				$$\Ch' +\Ce\sim rf^*\Ch,$$
				for an effective divisor $\Ce$. After passing to a stratification of $\Ct$, we may also assume $\Ce_t$ is effective for every $t\in\Ct$. We replace $\Cw$ by $\Cw'$, $W_i$ by $\Cw'_{t_i}$, $m$ by $rm$, $H_i$ by $\Ch'_{t_i}$ and $F_i$ by $rF_i+\Ce_{t_i}$, then we may assume $W_i$ is smooth and keep $h_i^*(mA_i)=g_i^*H_i+F_i$.
				
				Let $E_i,E'_i$ be the effective $h_i$-exceptional $\mathbb{Q}$-divisors on $Y_i$ such that
				$E_i,E'_i$ have no common component and
				$h_i^*K_{X_i}+E'_i=K_{Y_i}+E_i.$ Since $K_{X_i}+\delta_iA_i$ is big, $K_{Y_i}+E_i+\delta_ih_i^*A_i$ is big. Furthermore, its pushforward on $W_i$
				$$K_{W_i}+(g_i)_*E_i+\frac{\delta_i}{m}(H_i+(g_i)_*F_i)\sim_{\mathbb{Q}} (g_i)_*(K_{Y_i}+E_i+\delta_ih_i^*A_i)$$ 
				is also big. 
				
				Next we show that $(W_i,\Supp(H_i+(g_i)_*(E_i+F_i)))$ is log bounded and $\mathrm{coeff}(H_i+(g_i)_*F_i)$ is bounded from above.
				
				By the projection formula, we have $\mathrm{red}((g_i)_*(E_i+F_i)).H_i^{d-1}=\mathrm{red}(E_i+F_i).(g_i^*H_i)^{d-1}.$ Because $g_i^*H_i$ is base point free, by \cite[Lemma 3.2]{HMX13}, we have 
				\begin{equation*}
					\begin{aligned}
						& \mathrm{red}(E_i+F_i).(g_i^*H_i)^{d-1}\\
						\leq &2^d\mathrm{vol}(K_{Y_i}+\mathrm{red}(E_i)+g_i^*H_i+\lceil F_i\rceil)\\
						=&2^d\mathrm{vol}(K_{Y_i}+\mathrm{red}(E_i)+\lceil F_i\rceil-F_i+mh_i^*A_i).
					\end{aligned}
				\end{equation*}
				Because both $mA_i$ and $H_i$ are integral, then $\lceil F_i\rceil-F_i$ is exceptional over $X_i$. Also because
				$E_i$ is exceptional over $X_i$, then
				$$\mathrm{vol}(K_{Y_i}+\mathrm{red}(E_i)+\lceil F_i\rceil-F_i+mh_i^*A_i)\leq \mathrm{vol}(K_{X_i}+mA_i)\leq m^dv.$$
				Thus $\mathrm{red}(E_i+F_i).(g_i^*H_i)^{d-1}\leq (2m)^d v$. Because $H_i^d=\vol(H_i)\leq m^dv$, then by projection formula, $\mathrm{red}(H_i+(g_i)_*(E_i+F_i)).H_i^{d-1}\leq (2^d+1)m^d v$. Thus $(W_i,\Supp(H_i+(g_i)_*(E_i+F_i)))$ is log bounded.

				By log boundedness, to prove $\mathrm{coeff}(H_i+(g_i)_*F_i)$ is bounded from above, we only need to prove the degree $(H_i+(g_i)_*F_i).H_i^{d-1}=(g_i^*H_i+F_i).(g_i^*H_i)^{d-1}$ is bounded from above. Because $g_i^*H_i+F_i=h_i^*A_i$ is nef, then
				$$(g_i^*H_i+F_i).(g_i^*H_i)^{d-1}\leq (mh_i^*A_i+g_i^*H_i)^d=\mathrm{vol}(mh_i^*A_i+g_i^*H_i).$$
				Also because $g_i^*H_i\leq mh_i^*A_i$, then 
				$$\mathrm{vol}(mh_i^*A_i+g_i^*H_i)\leq \mathrm{vol}(2mh_i^*A_i)\leq (2m)^dv.$$
				Thus $\mathrm{coeff}(H_i+(g_i)_*F_i)$ is bounded from above.

				Since $(W_i,\Supp(H_i+(g_i)_*(E_i+F_i)))$ is bounded and $\mathrm{coeff}(H_i+(g_i)_*F_i)$ is bounded from above, we may assume there exists a reduced divisor $\Cb$ on $\Cw$ such that $H_i+(g_i)_*F_i\leq \Cb_{t_i}$.
				Let $\Cd$ be a divisor on $\Cw$ such that $\mathrm{red}((g_i)_*E_i)\leq  \Cd_{t_i}$ for every $i\in \mathbb{N}$. After passing to a stratification of $\Ct$, we may assume $(\Cw,\Cd)\rightarrow \Ct$ has a fiberwise log resolution $(\Cw',\Cd')\rightarrow \Ct$, where $\Cd'$ is the strict transform of $\Cd$ plus the exceptional divisor. Because $(\Cw',\Cd')$ is log smooth over $\Ct$, $(\Cw',(1-\epsilon)\Cd')$ is $\epsilon$-lc. Let $(\Cw'',\Cc)\rightarrow \Cw'$ be a fiberwise terminalization of $(\Cw',(1-\epsilon)\Cd')$, where $\Cc$ is the $\mathbb{Q}$-divisor such that $K_{\Cw''}+\Cc$ is equal to the pullback of $K_{\Cw'}+(1-\epsilon)\Cd'$. Define $\Cb'$ and $\Cb''$ to be the pullback of $\Cb$ on $\Cw'$ and $\Cw''$.
				
				Write $W_i'':=\Cw''_{t_i}$ and $W'_i:=\Cw'_{t_i}$. We may assume there are birational morphisms $g'_i:Y_i\rightarrow W'_i$ and $g''_i:Y_i\rightarrow W''_i$. Because $X_i$ is $\epsilon$-lc, then $E_i\leq (1-\epsilon)\mathrm{red}(E_i)$. 
				Also because 
				$K_{W'_i}+(g'_i)_*E_i+\delta_i(g'_i)_*h_i^*A_i$ is big, $(g'_i)_*E_i\leq (1-\epsilon)\mathrm{red}((g'_i)_*E_i)\leq (1-\epsilon)\Cd'_{t_i}$, and $(g'_i)_*h_i^*A_i\leq \Cb'_{t_i}$, then 
				$$K_{\Cw'_{t_i}}+(1-\epsilon)\Cd'_{t_i}+\delta_i \Cb'_{t_i}$$
				is big. 
				
				Because $h_i^*A_i$ is nef, $(g_i)_*h_i^*A_i\leq \Cb_{t_i}$, and $\Cb'_{t_i}$ is the pullback of $\Cb_{t_i}$ on $\Cw'_{t_i}$, then $h_i^*A_i\leq (g'_i)^*\Cb'_{t_i}$. Note adding exceptional divisor does not add global sections, then we have
				$$|r(K_{Y_i}+E_i+\delta_ih_i^*A_i)|\subset |r(K_{Y_i}+E_i+\delta_i(g'_i)^*\Cb'_{t_i})|\subset |r(K_{\Cw'_{t_i}}+(1-\epsilon)\Cd'_{t_i}+\delta_i \Cb'_{t_i})|$$
				for any $r\in \mathbb{N}$ sufficiently divisible. Also because $K_{\Cw''}+\Cc$ is equal to the pullback of $K_{\Cw'}+(1-\epsilon)\Cd'$ and $\Cb''$ equals to the pullback of $\Cb'$, we have
				$$|r(K_{\Cw'_{t_i}}+(1-\epsilon)\Cd'_{t_i}+\delta_i \Cb'_{t_i})|=|r(K_{\Cw''_{t_i}}+\Cc_{t_i}+\delta_i\Cb''_{t_i})|.$$
				Thus $|r(K_{Y_i}+E_i+\delta_ih_i^*A_i)|\subset |r(K_{\Cw''_{t_i}}+\Cc_{t_i}+\delta_i\Cb''_{t_i})|$, and then we have
				\begin{equation*}
					\begin{aligned}
						&|r(K_{Y_i}+E_i+\delta_ih_i^*A_i)|\\ 
						\subset& |r(K_{\Cw''_{t_i}}+\Cc_{t_i}\wedge (g''_i)_*E_i+\delta_i\Cb''_{t_i})| \\
						\subset& |r(K_{\Cw''_{t_i}}+\Cc_{t_i}\wedge (1-\epsilon)\mathrm{red}((g''_i)_*E_i)+\delta_i\Cb''_{t_i})|.
					\end{aligned}
				\end{equation*}
				
				Because $\Cc$ has finitely many components and $\mathrm{coeff}((1-\epsilon)\mathrm{red}((g''_i)_*E_i))=1-\epsilon$ is fixed, then after passing to a subsequence of $\{t_i,i\in \mathbb{N}\}$, we may assume there exists a $\mathbb{Q}$-divisor $\Cr\leq \Cc$ such that 
				$\Cc_{t_i}\wedge (1-\epsilon)\mathrm{red}((g'_i)_*E_i)=\Cr_{t_i}$ for every $i\in \mathbb{N}$.
				
				Because $K_{Y_i}+E_i+\delta_ih_i^*A_i$ is big for every $i\in\mathbb{N}$, then $K_{\Cw''_{t_i}}+\Cr_{t_i}+\delta_i\Cb''_{t_i}$ is big for every $i\in\mathbb{N}$. Since $\delta_i\rightarrow 0$, by invariance of plurigenera, after passing to a subsequence of $\{t_i\}$, $K_{\Cw''_{t_i}}+\Cr_{t_i}$ is pseudo-effective.
				
				Because $(\Cw''_{t_i},\Cc_{t_i})$ is terminal and $\Cr\leq \Cc$, then $(\Cw''_{t_i},\Cr_{t_i})$ is also terminal. Also because $\Cr_{t_i}\leq (g''_i)_*E_i$, then 
				$$(g''_i)^*(K_{\Cw''_{t_i}}+\Cr_{t_i})\leq K_{Y_i}+E_i,$$
				which means $K_{Y_i}+E_i$ is pseudo-effective. Because $K_{X_i}=(h_i)_*(K_{Y_i}+E_i)$, then $K_{X_i}$ is pseudo-effective, which is a contradiction since $K_{X_i}\sim_{\mathbb{Q}} -B_i$ is not pseudo-effective.
			\end{proof}
			
		\end{thm}

		\begin{thm}\label{boundedness of polarized CY in codimension 1}
			Fix $d,v\in\mathbb{N}$ and $\epsilon\in (0,1)$. Then the set $\mathscr{C}(d,v,\epsilon)$ is log bounded in codimension 1.
			\begin{proof}
				We follow the idea of the proof of \cite[Theorem 11.1]{Bir23b}. By induction on dimension, we may assume the result holds in dimension $d-1$. 
				
				After taking a $\mathbb{Q}$-factorization, we may assume $X$ is $\mathbb{Q}$-factorial.
				If $B=0$, then $(X,A)$ forms a log bounded family according to \cite[Theorem 1.5]{Bir23}.
				So we may assume $B\neq 0$. After running a $K_X$-MMP with scaling of $A$, we will have a birational contraction $g:X\dashrightarrow W$ and a Mori fiber space $h:W\rightarrow Z$. Because $(X,B)$ is $\epsilon$-lc Calabi--Yau and $W$ is $\mathbb{Q}$-factorial, $W$ is $\epsilon$-lc. 
				
				Let $t\in\mathbb{R}$ be the smallest number such that $K_X+tA$ is pseudo-effective, $B_W,A_W$ the pushforward of $B,A$ on $W$, then by the definition of Mori fiber space and Birkar-BAB theorem, we have $t\in \mathbb{Q}$ and
				$$K_W+tA_W\sim_{\mathbb{Q},Z} 0.$$
				By the canonical bundle formula, there exists a generalized pair $(Z,C+R)$ such that
				$$K_{W}+tA_W\sim_{\mathbb{Q}} h^*(K_Z+C+R).$$
				In particular, $K_Z+C+R$ is ample on $Z$.

				By Theorem \ref{boundedness of pseudo-effective threshold}, $t$ is larger than a positive number $\delta$ depending only on $d,v,\epsilon$. By the proof of \cite[Lemma 4.11]{Bir23}, $t$ is in a finite set depending only on $d,\epsilon$. Thus, there exists $p\in\mathbb{N}$ depending only on $t$ such that $p(K_{W}+tA_W)$ is integral. Because a general fiber $W_g$ of $h$ is $\epsilon$-lc Fano, $K_{W_g}+tA_{W_g}\sim_{\mathbb{Q}}0$ and $t$ is in a finite set, then by the Birkar-BAB theorem, $(W_g,tA_{W_g})$ is log bounded. Thus after replacing $p$ by a fixed multiple, we may assume $p(K_{W_g}+tA_{W_g})\sim 0$ and
				$$p(K_W+tA_W)\sim h^*(p(K_Z+C+R)).$$
				According to \cite[Corollary 1.3]{Bir23b}, the multiplicities of $h$ over codimension 1 points of $h$ is bounded, then after replacing $p$ by a fixed multiple, $p(K_Z+C+R)$ is an integral divisor. 
				
				Define $A_Z:=p(K_Z+C+R)$, next we show that there exists $v'>0$ such that $\mathrm{vol}(A_Z)\leq v'$. Let $\phi:V\rightarrow X,\psi:V\rightarrow W$ be a common resolution. By \cite[Theorem 1.1]{Bir23}, $|mA|$ defines a birational map for some fixed $m\in\mathbb{N}$. Let $c:=\mathrm{dim}(Z)$, then $(m\phi^* A|_{V_g})^{d-c}\geq 1$, where $V_g$ is a general fiber of $V\rightarrow Z$. Thus $(m\phi^* A)^{d-c}$ is an effective cycle of codimension $d-c$ whose degree on general fibers of $V\rightarrow T$ is $\geq 1$. By the projection formula, we have
				$$(m\phi^* A)^{d-c}.(\psi^*h^* A_Z)^c\geq (A_Z)^c.$$
				
				Because both $A$ and $A_Z$ are ample, then
				\begin{equation*}
					\begin{aligned}
						& (m\phi^* A)^{d-c}.(\psi^*h^* A_Z)^c\\
						\leq & (m\phi^*A+\psi^*h^* A_Z)^d\\
						\leq & \mathrm{vol}(pK_X+(pt+m)A)^d\\
						\leq & (pt+m)^d\mathrm{vol}(A)\\
						\leq & (pt+m)^dv.
					\end{aligned}
				\end{equation*}
				where the second to last inequality comes from $K_X\sim_{\mathbb{Q}}-B$. 
				Then $\mathrm{vol}(A_Z)\leq v'$ for a fixed number $v'>0$ depending only on $d,\epsilon,v$.
				
				By \cite[Theorem 1.1]{Bir23b}, there exists $\epsilon'\in (0,1)$ depending only on $d,\epsilon$ such that $(Z,C+R)$ is generalized $\epsilon'$-lc. Because the moduli part is $\mathbf{b}$-nef and abundant, see \cite{Amb05}, there exists a $\mathbb{Q}$-divisor $C'$ on $Z$ such that $C'\sim_{\mathbb{Q}} C+R$ and $(Z,C')$ is $\frac{\epsilon'}{2}$-lc. 
				Note we assume the result in dimension $d-1$. Because $(Z,C')$ is $\frac{\epsilon'}{2}$-lc and  $A_Z$ is ample and integral and $\mathrm{vol}(A_Z)\leq v'$, then $(Z,A_Z)$ is log bounded in codimension one. We assume there exists $r\in \mathbb{N}$ depending only on $d,\epsilon,v$, a projective variety $Z'$ and a very ample divisor $H_{Z'}$ on $Z'$ such that
				\begin{itemize}
					\item $Z'$ is isomorphic in codimension one with $Z$, 
					\item $H_{Z'}-A_{Z'}$ is ample, where $A_{Z'}$ is the strict transform of $A_Z$, and
					\item $\vol(H_{Z'})\leq r$.
				\end{itemize}
				
				Because $Z\dashrightarrow Z'$ is an isomorphic in codimension one, $(W,B_W)$ is $\epsilon$-lc, and $K_W+B_W\sim_{\mathbb{Q}} 0$, by \cite[Proposition 3.7]{BDCS20}, there exists a projective pair $(W',B_{W'})$ and a contraction $W'\rightarrow Z'$ such that $(W',B_{W'})\dashrightarrow(W,B_W)$ is an isomorphism in codimension one.
				
				Since $(X,B)$ is klt Calabi--Yau and $X\dashrightarrow W$ is a birational contraction, $(X,B)$ is crepant birationally equivalent to $(W,B_W)$, and hence to $(W',B_{W'})$. Let $X'\rightarrow W'$ be a birational morphism extracting precisely the exceptional divisor of $X\dashrightarrow W$, and let $B'$ be the strict transform of $B$ on $X'$. Then $(X,B)\dashrightarrow (X',B')$ is an isomorphism in codimension one. Thus we have the following diagram.
				$$\xymatrix{
					X \ar@{-->}[r]^g \ar@{-->}[d]  & W\ar[r]^h \ar@{-->}[d] & Z\ar@{-->}[d]\\
					X'\ar[r]_{g'} & W'\ar[r]_{h'} & Z'. 
				}$$

				Because $(W',B_{W'})$ is $\epsilon$-lc, $K_{W'}+B_{W'}\sim_{\mathbb{Q}} 0$, and $B_{W'}$ is big over $Z'$, then $h':(W',B_{W'})\rightarrow Z'$ is a $(d,r,\epsilon)$-Fano type contraction. Thus by \cite[Theorem 1.2]{Bir22}, $W'$ is bounded. We may assume there exists $r'$ depending only on $d,r,\epsilon$ and a very ample divisor $H_{W'}$ on $W'$ such that 
				\begin{itemize}
					\item $H_{W'}-h'^*H_{Z'}$ is ample, 
					\item $H_{W'}+K_{W'}$ is pseudo-effective, and
					\item $\vol(H_{W'})\leq r'$.
				\end{itemize}
				
				Because $H_{Z'}-A_{Z'}$ and $H_{W'}-h'^*H_{Z'}$ is ample, then $H_{W'}-h'^*A_{Z'}$ is ample. Since $p(K_W+tA_W)\sim h^*A_Z$, then $p(K_{W'}+tA_{W'})\sim h'^*A_{Z'}$, where $A_{W'}$ is the strict transform of $A_W$ on $W'$. Since $W'$ is bounded and $l,t$ are in a finite set depending only on $d,\epsilon$, after replacing $H_{W'}$ by a fixed multiple, we may assume $H_{W'}-A_{W'}$ is pseudo-effective. Thus $(W',A_{W'})$ is log bounded. 
				
				Because both $H_{W'}-A_{W'}$ and $H_{W'}+K_{W'}$ are pseudo-effective, by \cite[Theorem 1.8]{Bir21}, there exists $t'$ depending only on the log bounded set, such that $(W',B'+2t'A_{W'})$ is lc, then $(W',B'+t'A_{W'})$ is $\frac{\epsilon}{2}$-lc. 
				
				Since $K_{X'}+B'\sim_{\mathbb{Q}} g'^*(K_{W'}+B_{W'})$ and $(W',B'+t'A_{W'})$ is $\frac{\epsilon}{2}$-lc, then $(X',B'+t'g'^*A_{W'})$ is $\frac{\epsilon}{2}$-lc. Because $t'$ is fixed, after replacing $H_{W'}$ by a fixed multiple, we may assume $H_{W'}-t'A_{W'}$ is ample. Thus, $(X',B'+t'g'^*A_{W'})\rightarrow W'$ is a $(d,r',\frac{\epsilon}{2})$-Fano fibration. By \cite[Theorem 1.2]{Bir22}, $(X',h'^*A_{X'})$ is log bounded. Let $A'$ be the strict transform of $A$ on $X'$. Since $A$ is nef, by negativity lemma, $A'\leq h'^*A_{X'}$. Therefore, $(X',A')$ is log bounded and $(X,A)$ is log bounded in codimension one.
			\end{proof} 
		\end{thm}
		
		\begin{lemma}\label{e-lc threshold}
			Fix $\epsilon\in (0,1)$. Let $\mathscr{S}$ be a log bounded set of couples. Then there exists $\delta\in (0,1)$ depending only on $\epsilon$ and $\mathscr{S}$ satisfying the following
			
			If $(X,A)\in \mathscr{S}$ is a couple, $B$ is a $\mathbb{Q}$-divisor on $X$ such that
			$(X,B)$ is a $\epsilon$-lc Calabi--Yau pair, then $(X,B+\delta D)$ is $\frac{\epsilon}{2}$-lc for every $D\in |A|_{\mathbb{Q}}$.
			\begin{proof}
				Suppose every couple in $\mathscr{S}$ has dimension $d$.
				
				By the definition of log boundedness, there exists $r\in\mathbb{N}$ depending only on $\mathscr{S}$ such that for every $(X,A)\in \mathscr{S}$, there exists a very ample divisor $H$ on $X$ such that
				\begin{itemize}
					\item $\vol(H)\leq r$, and
					\item $H-A$ is pseudo-effective.
				\end{itemize}
				By boundedness, after replacing $H$ and $r$ by fixed multiples, we may also assume $H+K_X$ is pseudo-effective.
				
				If $B$ is a $\mathbb{Q}$-divisor on $X$ such that $(X,B)$ is a $\epsilon$-lc Calabi--Yau pair, then $B\sim_{\mathbb{Q}} -K_X$ and $H-B\sim_{\mathbb{Q}} H+K_X$ is pseudo-effective. By \cite[Theorem 1.8]{Bir21}, there exists $\delta'\in (0,1)$ depending only on $d,r,\epsilon$ such that
				$$(X,B+\delta' D)\text{ is lc for every }D\in |A|_{\mathbb{Q}}.$$
				Define $\delta:=\frac{\delta'}{2}$. Since $(X,B)$ is $\epsilon$-lc, then $(X,B+\delta D)$ is $\frac{\epsilon}{2}$-lc for every $D\in |A|_{\mathbb{Q}}$.
			\end{proof}
		\end{lemma}
		\begin{thm}\label{log boundedness of PCY}
			Fix $d,v\in \mathbb{N} $ and $ \epsilon\in (0,1)$. Then the set $\mathscr{C}(d,v,\epsilon)$ is log bounded.
			\begin{proof}
				By Theorem \ref{boundedness of polarized CY in codimension 1}, $\mathscr{C}(d,v,\epsilon)$ is log bounded in codimension one. Then there exists a log bounded set of projective couples $\mathscr{S}$ such that for every $(X,A)\in \mathscr{C}(d,v,\epsilon)$, there exists $(X',A')\in \mathscr{C}$ and an isomorphism in codimension one $(X,A)\dashrightarrow (X',A')$.
				
				Suppose $B$ is an effective $\mathbb{Q}$-divisor on $X$ such that $(X,B)$ is $\epsilon$-lc Calabi--Yau. Let $B'$ be the strict transform of $B$ on $X'$, then $(X',B')$ is also an $\epsilon$-lc Calabi--Yau pair. By Lemma \ref{e-lc threshold}, there exits $\delta>0$ depending only on $d,v,\epsilon$ such that $(X',B'+\delta D')$ is $\frac{\epsilon}{2}$-lc for every $D'\in |A'|_{\mathbb{Q}}$.
				
				Because $X'\dashrightarrow X$ is an isomorphism in codimension one, a $\mathbb{Q}$-divisor $D'$ is in $ |A'|_{\mathbb{Q}}$ if and only if its strict transform $D$ is in $|A|_{\mathbb{Q}}$. And by negativity lemma, because $A$ is ample, we have $(X,B+\delta D)$ is $\frac{\epsilon}{2}$-lc for every $D\in |A|_{\mathbb{Q}}$.
				
				Because $A$ is integral ample and $(X,B)$ is $\epsilon$-lc, by \cite[Theorem 1.1]{Bir23}, there exists $m\in\mathbb{N}$ depending only on $d,\epsilon$ such that $|mA|$ defines a birational map. In particular $\vol(mA)\geq 1$, then $\vol(A)\geq \frac{1}{m^d}$.
				By Lemma \ref{boundedness of Cartier index by volume and lct}, there exists $r\in \mathbb{N}$ depending only on $d,\delta,m$ such that the Cartier index of any integral divisor on $X$ is a factor of $r$. In particular, $rA$ is Cartier. Then by \cite[Theorem 1.1 and Lemma 1.2]{Kol93}, there exists $r'$ depending only on $d,r$ such that $r'A$ is very ample. Because $\vol(A)\leq v$, then $(X,A)$ is log bounded.
			\end{proof}
		\end{thm}
		\begin{proof}[{Proof of Theorem \ref{Main theorem: discreteness of volume}}]
			Fix $v>0$, we only need to show that if $(X,B)$ is a $d$-dimensional $\epsilon$-lc Calabi--Yau pair and $A$ is a big integral divisor on $X$ such that $\vol(A)\leq v$, then $\vol(v)$ is in a finite set depending only on $d,\epsilon,v$. After taking a $\mathbb{Q}$-factorization of $X$, we may assume $X$ is $\mathbb{Q}$-factorial and $A$ is $\mathbb{Q}$-Cartier.
			
			Choose $D\in|A|_{\mathbb{Q}}$ and $0<\delta\ll 1$ such that $(X,B+\delta D)$ is klt, let $\phi:X\dashrightarrow Y$ be the canonical model of $(X,B+\delta D)$. Because $(X,B)$ is Calabi--Yau, then $(Y,B_Y)$ is Calabi--Yau and
			$$\vol(A)=\vol(D)=\vol(D_Y)=\vol(A_Y),$$
			where $B_Y,D_Y$ and $A_Y$ are the pushforward of $B,D$ and $A$. We may replace $(X,B),A$ by $(Y,B_Y),A_Y$ and assume $A_Y$ is ample. Since $\vol(A_Y)\leq v$, the result follows from Theorem \ref{log boundedness of PCY}.
		\end{proof}
		
		\begin{proof}[{Proof of Theorem \ref{Main theorem: boundedness of polarized Calabi--Yau}}]
			This follows directly from Theorem \ref{log boundedness of PCY}.
		\end{proof}
		
		\noindent\textbf{Acknowledgments}. The author would like to thank his postdoc mentor Caucher Birkar for his encouragement and constant support. The author also acknowledges Jingjun Han, Guodu Chen, and Xiaowei Jiang for their valuable comments. This work was supported by BMSTC and ACZSP. 	
	
\end{document}